\documentclass[12pt, reqno, twoside, letterpaper]{amsart}

\usepackage{amsmath,amssymb,amsbsy,amsfonts,amsthm,latexsym,
            amsopn,amstext,amsxtra,euscript,amscd,stmaryrd,mathrsfs,
            cite,array,comment}

\numberwithin{equation}{section}

\usepackage{color}
\usepackage[usenames,dvipsnames,svgnames,table]{xcolor}

\def\eps{\varepsilon}
\def\mand{\qquad\mbox{and}\qquad}
\def\fl#1{\left\lfloor#1\right\rfloor}

\def\({\left(}
\def\){\right)}

\newcommand{\e}{\ensuremath{\mathbf{e}}}

\newcommand{\cB}{\ensuremath{\mathcal{B}}}

\newcommand{\cI}{\ensuremath{\mathcal{I}}}

\newcommand{\cN}{\ensuremath{\mathcal{N}}}

\newcommand{\cS}{\ensuremath{\mathcal{S}}}

\newcommand{\cT}{\ensuremath{\mathcal{T}}}

\newcommand{\fS}{\ensuremath{\mathfrak{S}}}

\newcommand{\RR}{\ensuremath{\mathbb{R}}}

\usepackage[text={144mm,233mm},centering]{geometry} 

\usepackage{amsthm}
\newtheoremstyle{customthm}
{1em}                    
{1em}                    
{\itshape}               
{}                       
{\scshape}               
{.}                      
{5pt plus 1pt minus 1pt} 
{}                       

\newtheoremstyle{customrem}
{1em}                    
{1em}                    
{}                       
{}                       
{\scshape}               
{.}                      
{5pt plus 1pt minus 1pt} 
{}                       

\theoremstyle{customthm}

\newtheorem{X}{X}[section]
\newtheorem{theorem}[X]{Theorem}  
  
\newtheorem{lemma}[X]{Lemma}
\newtheorem{corollary}[X]{Corollary}

\theoremstyle{customrem}

\usepackage{etoolbox}
\AtEndEnvironment{remark}{\null\hfill\qedsymbol}
\newtheorem{definition}[X]{Definition}
\AtEndEnvironment{definition}{\null\hfill\qedsymbol}

\renewcommand{\le}{\ensuremath{\leqslant}}
\renewcommand{\ge}{\ensuremath{\geqslant}}

\renewcommand{\pod}[1]{\mathchoice
  {\allowbreak \if@display \mkern 5mu\else \mkern 5mu\fi (#1)}
  {\allowbreak \if@display \mkern 5mu\else \mkern 5mu\fi (#1)}
  {\mkern4mu(#1)}
  {\mkern4mu(#1)}
}

\DeclareSymbolFont{EUEX}{U}{euex}{m}{n}

\DeclareSymbolFont{euexlargesymbols}{U}{euex}{m}{n}
\DeclareMathSymbol{\intop}{\mathop}{euexlargesymbols}{"52}
     \def\int{\intop\nolimits}

\DeclareSymbolFont{euexsymbols}     {U}{euex}{m}{n}
\DeclareMathSymbol{\smallint}{\mathop}{euexsymbols}{"52}
                     
\allowdisplaybreaks

\title[Exponential sums related to Piatetski-Shapiro primes]
      {Improvements on exponential sums related to Piatetski-Shapiro primes}
        
\author[Li Lu]{Li Lu}

\address{School of Mathematics and Statistics, Xi'an Jiaotong University, Xi'an,Shaanxi,China.}

\email{lilu\_math@foxmail.com}

\author[Lingyu Guo]{Lingyu Guo}

\address{School of Mathematics and Statistics, Xi'an Jiaotong University, Xi'an,Shaanxi,China.}

\email{guo.lingyu@foxmail.com }

\author[Victor Zhenyu Guo]{Victor Zhenyu Guo}

\address{School of Mathematics and Statistics, Xi'an Jiaotong University, Xi'an,Shaanxi,China.}

\email{guozyv@xjtu.edu.cn; vzguo@foxmail.com}

\date{\today}

\begin{document}

\setlength{\extrarowheight}{5pt}

\begin{abstract}
We prove a new bound to the exponential sum of the form
$$
\sum_{h \sim H}\delta_h \mathop{\sum_{m\sim M}\sum_{n\sim N}}_{mn\sim x}a_{m}b_{n}\e\big(\alpha mn + h(mn + u)^{\gamma}\big),
$$
by a new approach to the Type I sum. The sum can be applied to many problems related to Piatetski-Shapiro primes, which are primes of the form $\lfloor n^c \rfloor$. In this paper, we improve the admissible range of the Balog-Friedlander condition, which leads to an improvement to the ternary Goldbach problem with Piatetski-Shapiro primes. We also investigate the distribution of Piatetski-Shapiro primes in arithmetic progressions, Piatetski-Shapiro primes in a Beatty sequence and so on. 
\end{abstract}

\maketitle

\begin{quote}
\textbf{MSC Numbers:} 11N05; 11B83; 11L07.
\end{quote}

\begin{quote}
\textbf{Keywords:} exponential sums; Piateski-Shapiro primes; ternary Goldbach problem.
\end{quote}

\newcommand{\tind}[1]{\ensuremath{\widetilde{\mathbf{1}}_{#1}}} 


\section{Introduction and main results}
\label{Intro}

The Piatetski-Shapiro sequences are sequences of the form
$$
\cN^{(c)} = (\lfloor n^{c} \rfloor)_{n=1}^\infty,
$$
where $\fl{\cdot}$ is the integer part. Piatetski-Shapiro~\cite{PS} proved the Piatetski-Shapiro prime number theorem stating that for $ 1 < c < \frac{12}{11}$ the counting function
$$
\pi^{(c)}(x) = \# \big\{\text{\rm prime~} p\le x : p \in \cN^{(c)} \big\}
$$
satisfies the asymptotic relation
\begin{equation}
\label{eq:PS}
\pi^{(c)}(x) = (1 + o(1)) \frac{x^{1/c}}{\log x} \qquad \text{~\rm as } x \to \infty.
\end{equation}
The admissible range for $c$ of the above formula has been extended many times and is currently known to hold for all $1 < c < \frac{2817}{2426}$ thanks to Rivat and Sargos~\cite{RiSa}. Rivat and Wu~\cite{RiWu} also showed that there are infinitely many Piatetski-Shapiro primes for $1 < c < \frac{243}{205}$ without an asymptotic formula. We refer the readers to~\cite{Guo2} for more details of the improvements of $c$. The asymptotic relation is expected to hold for all values of $1 < c < 2$. The estimation of Piatetski-Shapiro primes is an approximation of the well-known conjecture that there exist infinitely many primes of the form $n^2+1$. 

In this article, we mainly consider three types of problems related to Piatetski-Shapiro primes by proving a new bound on 
an special type of exponential sums which improves the bound by Kumchev \cite{Kumc}. Several results are improved based on our new estimation of the exponential sum. Throughout the paper, we always denote $\gamma = c^{-1}$ for any subscripts. 

\subsection{The ternary Goldbach problem on Piatetski-Shapiro primes}
A natural problem to study prime numbers is the ternary Goldbach problem, which is proved by Vinogradov that every odd sufficiently large integer is a sum of three primes. In 1992, Balog and Friedlander \cite{BaFr} considered the ternary Goldbach problem on Piatetski-Shapiro primes and proved that for every sufficiently large $N$
\begin{equation}
\label{eq:PS3}
\sum_{\substack{p_1 + p_2 + p_3 = N \\ p_i \in \mathcal{N}^{(c_i)}}} 1 = (1+o(1)) \frac{\gamma_1\gamma_2\gamma_3 \Gamma(\gamma_1) \Gamma(\gamma_2) \Gamma(\gamma_3)}{\Gamma(\gamma_1 + \gamma_2 + \gamma_3)} \frac{\mathfrak{S}(N)N^{\gamma_1 + \gamma_2 + \gamma_3 -1}}{\log^3 N}, 
\end{equation}
where $\Gamma(\cdot)$ is the Gamma function and $\mathfrak{S}(N)$ is the singular series which is defined as
\begin{equation}
\label{eq:ss}
\mathfrak{S}(N) = \prod_{p | N} \bigg(1-\frac{1}{(p-1)^2} \bigg) \prod_{p \nmid N} \bigg(1 + \frac{1}{(p-1)^3} \bigg),
\end{equation}
with $1 < c_1, c_2, c_3 < \frac{21}{20}$. The admissible range of $c$ was improved several times as the following table. 

\begin{center}\label{tab:3p}
	\begin{tabular}{||c|c||}
		\hline
		\vphantom{\Big|} Authors  &  The upper bound of $c$ \\
		\hline
		Balog and Friedlander \cite{BaFr} & $\frac{21}{20} = 1.05$  \\
		Rivat \cite{Riva} & $\frac{199}{188} = 1.0585...$  \\
		Kumchev \cite{Kumc} & $\frac{53}{50} = 1.06$  \\
		Sun, Pan and Du \cite{SDP} & $\frac{73}{64} = 1.1406...$ \\ [0.03in]
		\hline
	\end{tabular}
\end{center}

We improve the admissible range of $c$ and prove the following theorem. 

\begin{theorem}
\label{thm:3ps}
With $1 < c_1, c_2, c_3 < \frac{569}{498} = 1.1425\dots$, \eqref{eq:PS3} holds. 
\end{theorem}

To prove Theorem \ref{thm:3ps}, the key part is the Balog-Friedlander condition due to Sun, Du and Pan. 

\begin{definition} [Balog-Friedlander condition]
\label{def:BF}
A real number $c \in (1,2)$ satisfies the Balog-Friedlander condition if for every sufficiently large $N$, it follows that
\begin{equation}
\label{eq:BF}
\sum_{\substack{p \le N \\ p \in \mathcal{N}^{(c)}}} c p^{1-\frac{1}{c}}\log p \cdot \e(\alpha p) = \sum_{p \le N} \log p\cdot \e(\alpha p) + O ( N^{1-\eps} ),
\end{equation}
uniformly for $\alpha$. 
\end{definition}

Balog and Friedlander proved that the condition holds for $1 < c < \frac{9}{8} = 1.125$ while Kumchev improved it to $1 < c < \frac{73}{64} = 1.1406...$. We obtain the following theorem.

\begin{theorem}
\label{thm:wb}
For $1 < c < \frac{569}{498} = 1.1425\dots$, the Balog-Friedlander condition holds. 
\end{theorem}

We also provide a more general form of this theorem in Section \ref{sec:8}. 


\subsection{Piatetski-Shapiro primes in arithmetic progressions}
Another natural topic related to primes is the study of primes in arithmetic progressions. For any integers $a$ and $q$ with $(a,q)=1$, we define the counting function
\begin{equation*}
\pi^{(c)} (x; q, a) = \#\big\{ p \leqslant x: p \equiv a \pmod q, \, p \in \mathcal{N}^{(c)} \big\}
\end{equation*}
in terms of the more familiar function
\begin{equation*}
\pi(x; q, a) = \#\big\{p \leqslant x: p \equiv a \pmod q \big\}.
\end{equation*}
Baker, Banks, Br\"{u}dern, Shparlinski and Weingartner \cite{BBBSW} proved an asymptotic formula and a bound of error term to $\pi^{(c)} (x; q, a)$. We list all the investigations of the admissble range of $c$ in the following table. 

\begin{center}\label{tab:ap}
	\begin{tabular}{||c|c||}
		\hline
		\vphantom{\Big|} Authors  &  The range of $c$ \\ 
		\hline 
		Baker et al. \cite{BBBSW} & $1 < c < \frac{18}{17} = 1.0588\dots $  \\ 
		Guo \cite{Guo1} & $1 < c < \frac{14}{13} = 1.0769\dots $  \\ 
		Guo, Li and Zhang \cite{GLZ2} & $1 < c < \frac{12}{11} = 1.0909\dots$  \\ [0.03in]
		\hline
	\end{tabular}
\end{center}

In this paper, we continue improving the range of $c$ and establish the following result. 

\begin{theorem}\label{thm:AP}
Let $a$ and $q$ be coprime integers. For $1 < c < \frac{569}{498} = 1.1425\dots$, we have
$$
\pi^{(c)}(x;q,a) = \gamma x^{\gamma-1} \pi(x;q,a)+\gamma(1-\gamma)\int_2^xu^{\gamma-2}\pi(u;q,a)\mathrm{d}u+O(x^{\gamma - \eps}),
$$
where the implied constant depends on $\gamma$ and $\eps$. 
\end{theorem}

\subsection{Piatetski-Shapiro primes in a Beatty sequence}
For fixed real numbers $\alpha$ and $\beta$, the associated non--homogeneous Beatty sequence is the sequence of integers defined by
\[
\cB_{\alpha,\beta} = \(\fl{\alpha n+\beta}\)_{n=1}^\infty,
\]
which are also called generalized arithmetic progressions. Guo, Li and Zhang \cite{GLZ} proved that there are infinitely many Piatetski-Shapiro primes in the Beatty sequence with some certain restrictions for $1 < c < \frac{12}{11}$.  With a new estimation of the related exponential sums, we extend the admissible range of $\gamma$. 

\begin{theorem}
\label{thm:3}
Suppose $\beta\in \mathbb{R}$. Let $\alpha > 1$ be irrational and of finite type.
For $1 < c < \frac{569}{498} = 1.1425\dots$, there are infinitely many primes in the intersection of Beatty sequence $\cB_{\alpha,\beta}$ and the Piatetski--Shapiro sequence $\cN^{(c)}$.
Moreover, the counting function
$$
\pi^{(c)}_{\alpha,\beta} (x) = \# \big\{ \text{\rm prime~}p \le x: p \in \cB_{\alpha,\beta} \cap \cN^{(c)} \big\}
$$
satisfies
$$
\pi^{(c)}_{\alpha,\beta} (x)  = \frac{x^{\gamma}}{{\alpha} \log x} +O\bigg(\frac{x^{\gamma}}{\log^2x}\bigg),
$$
where the implied constant depends only on $\alpha$ and $\gamma$.
\end{theorem}

\subsection{A key result to exponential sums}
The improvements of our theorems related to Piatetski-Shapiro sequences rely on the following bound. The definition of exponent pair can be found on Page 30 of \cite{GraKol}.

\begin{theorem}
\label{thm:eps}
Let $(k,l)$ be an exponent pair such that
$$
4k-2l+1>0 
$$
and $\gamma$ subjected to
\begin{align*}
\max\biggl\{\frac{13}{15},  \frac{12k+10}{12k-2l+13} \biggr\} <\gamma<1.
\end{align*}
Assume further that $H\leqslant x^{1-\gamma + \eps}$ and $0\leqslant u \leqslant 1$. Then
\begin{align*}
	\sum_{h\sim H}\bigg| \sum_{n \sim x}\Lambda(n)\e\big(\alpha n + h(n+u)^{\gamma}\big) \bigg| \ll x^{1-\eps}.
\end{align*}
\end{theorem}

Thanks to Bourgain's powerful exponent pair $(\frac{13}{84}+\eps, \frac{55}{84}+\eps)$ in Theorem~6 of \cite{Bour}, we obtain the following corollary.
\begin{corollary}
\label{cor:eps}
\begin{align*}
	\sum_{h\sim H}\bigg| \sum_{n \sim x}\Lambda(n)\e\big(\alpha n + h(n+u)^{\gamma}\big) \bigg| \ll x^{1-\eps}
\end{align*}
provided $\frac{498}{569}<\gamma < 1$,~ $H\leqslant x^{1-\gamma + \eps}$ and $0\leqslant u\leqslant 1$.
\end{corollary}

\subsection{Other related topics}

We should mention that our Corollary \ref{cor:eps} is very useful while people are investigating combination problems related to Piatetski-Shapiro primes, Beatty primes, primes in arithmetic progressions and the ternary Goldbach problem. Indeed, we believe that following questions can also be proved with the same admissible range: 

\begin{itemize}
\item Piatetski-Shapiro primes in arithmetic progressions in a Beatty sequence;
\item The ternary Goldbach problem with primes in the intersection of a Piatetski-Shapiro sequence and Beatty sequences;
\item The ternary Goldbach problem with Piatetski-Shapiro primes in arithmetic progressions;
\item and so on.
\end{itemize}

However, to make the article short, we will not discuss these problems. 

\section{Preliminary}

\subsection{Notations}

We denote by $\fl{t}$ and $\{t\}$ the integer part and the fractional part of $t$, respectively. As is customary, we put $\e(t)= e^{2\pi it}$. We make considerable use of the sawtooth function defined by
$$
\psi(t) = t-\fl{t}-\frac{1}{2}=\{t\}-\frac{1}{2}\qquad(t\in\RR).
$$

The letter $p$ always denotes a prime. For the Piatetski-Shapiro sequence $(\fl{n^c})_{n=1}^\infty$, we denote $\gamma = c^{-1}$. $\delta$ is a parameter satisfies $0\leqslant \delta \leqslant 1-\gamma$. $\eps$ is always a sufficiently small positive number. $x\sim X$ means that $x$ runs through a subinterval of $(X,2X]$, which endpoints are not necessary the same in different equations and may depend on outer summation variables.
	
Throughout the paper, implied constants in symbols $O$, $\ll$ and $\gg$ may depend (where obvious) on the parameters $c, \eps$ but are absolute otherwise. For given functions $F$ and $G$, the notations $F\ll G$, $G\gg F$ and $F=O(G)$ are all equivalent to the statement that the inequality $|F|\le C|G|$ holds with some constant $C>0$. $F \asymp G$ means that $F \ll G \ll F$.

\subsection{Technical lemmas}

We need the following approximation of Vaaler \cite{Vaal}.

\begin{lemma}
\label{lem:Vaaler}
For any $H\ge 1$ there are numbers $a_h,b_h$ such that
$$
\bigg|\psi(t)-\sum_{0<|h|\le H}a_h\,\e(th)\bigg| \le\sum_{|h|\le H}b_h\,\e(th),\qquad a_h\ll\frac{1}{|h|}\,,\qquad b_h\ll\frac{1}{H}\,.
$$
\end{lemma}

We use the following well-known lemma, which provides a characterization of the numbers that occur in the Piatetski-Shapiro sequence $\cN^{(c)}$; see \cite[Lemma 2.7]{GLZ}.

\begin{lemma}
	\label{lem:PS}
	A natural number $m$ has the form $\fl{n^c}$ if and only if $\mathbf{1}^{(c)}(m) = 1$, where
	$\mathbf{1}^{(c)}(m) = \fl{-m^\gamma} - \fl{-(m+1)^\gamma}$.  Moreover,
	$$
	\mathbf{1}^{(c)}(m)=\gamma m^{\gamma-1}+ \psi(-(m+1)^\gamma) - \psi(-m^\gamma)+O(m^{\gamma-2}).
	$$
\end{lemma}

The following lemma is the famous Weyl-van der Corput inequality, also called the A-process; see \cite[Lemma 2.5]{GraKol}. 

\begin{lemma}
\label{lem:A}
Suppose $f(n)$ is a complex valued function and $I$ is an interval such that $f(n) = 0$ if $n \notin I$. If $H$ is a positive integer then
$$
|\sum_{n \in I} f(n)|^2 \le \frac{|I| + H}{H} \sum_{|h|< H} \big(1 - \frac{h}{H} \big) \sum_{n \in I} f(n) \overline{f(n-h)}.  
$$
\end{lemma}

The following lemma is derived from the Poisson summation formula, known as the B-process; see \cite[Lemma 3.6]{GraKol}.

\begin{lemma}
\label{lem:B}
Suppose that $f$ has four continuous derivatives on $[a,b]$ and that $f''<0$ on this interval. Suppose further that $[a,b] \subset [N, 2N]$ and that $\alpha = f'(b)$ and $\beta = f'(a)$. Assume that there is some $F>0$ such that 
$$
f^{(2)}(x) \asymp FN^{-2}, f^{(3)}(x) \ll FN^{-3}, \text{~and~} f^{(4)}(x) \ll FN^{-4} 
$$
for $x$ in $[a,b]$. Let $x_{\nu}$ be defined by the relate $f'(x_\nu) = \nu$, and let $\phi(\nu) = -f(x_\nu) + \nu x_\nu$. Then
$$
\sum_{n \in I} \e\big(f(n)\big) = \sum_{\alpha \le \nu \le \beta} \frac{\e\big(-\phi(\nu) - 1/8\big)}{|f''(x_\nu)|^{1/2}} + O\big(\log(FN^{-1} + 2) + F^{-1/2} N\big). 
$$
\end{lemma}

We need the following lemma by van der Corput; see~\cite[Theorem ~2.2]{GraKol}.

\begin{lemma}
\label{lem:GK3}
Let $f$ be twice continuously differentiable on a subinterval $\cI$ of $(N,2N]$. Suppose that for some $\lambda>0$, the inequalities
$$
\lambda\ll|f''(t)|\ll\lambda\qquad(t\in\cI)
$$
hold, where the implied constants are independent of $f$ and $\lambda$. Then
$$
\sum_{n\in\cI}\e\big(f(n)\big)\ll N\lambda^{1/2}+\lambda^{-1/2}.
$$
\end{lemma}

\section{Outline}
\label{sec:3}

We prove Theorem~\ref{thm:eps} about the bound of the key exponential sum in Section \ref{sec:4}, which is the essential reason we improve the results related to Piatetski-Shapiro primes. We prove the Vinogradov theorem with Piatetsk-Shapiro primes and the Balog-Friedlander condition in Section \ref{sec:5}. The proof of Piatetski-Shapiro primes in arithmetic progressions is in Section \ref{sec:6} while a sketch of Theorem \ref{thm:3} is in Section \ref{sec:7}. In section \ref{sec:8} we provide a more general Balog-Friedlander condition and the related bound of exponential sum due to historical record. 

\subsection{The method of exponential sums}

We mention the remarkable result by Heath-Brown \cite{HB2}, who proved that the Piatetski-Shapiro prime number theorem (equation \eqref{eq:PS}) holds for 
\begin{equation}
\label{eq:HBP}
\max\biggl\{\frac{13}{15},  \frac{12k+10}{12k-2l+13} \biggr\} < \gamma < 1
\end{equation}
and some other restrictions to $k, l, \gamma$, where $(k,l)$ is an exponent pair; for more details see Section 7 in \cite{HB2}. Heath-Brown investigated the exponential sum
\begin{equation}
\label{eq:oes}
\sum_{h\sim H}\bigg| \sum_{n \sim x}\Lambda(n)\e( h n^{\gamma}) \bigg|
\end{equation}
by splitting the sum into Type I and Type II sums via the Heath-Brown identity. He achieved the best bound (up to now) of the Type II sum and calculated the Type I sum by the method of exponent pair. 

We recall that the exponential sum we need to consider is a more complicated form
\begin{equation}
\label{eq:nes}
\sum_{h\sim H}\bigg| \sum_{n \sim x}\Lambda(n)\e\big(\alpha n + h(n+u)^{\gamma}\big) \bigg|. 
\end{equation}
The sum can also be partitioned into a combination of corresponding Type I and Type II sums by the Heath-Brown identity. Balog and Friedlander \cite{BaFr} obtained a same upper bound of the Type II sum as Heath-Brown \cite{HB2}. However, since the function
$$
\alpha n + h(n+u)^{\gamma}
$$
does not satisfies the prerequisite of the method of exponent pair, people cannot prove a bound of the Type I sum as the same strength as Heath-Brown. 

In this paper, we apply Lemma \ref{lem:A} (A-process) and Lemma \ref{lem:B} (B-process), plus a technical approach to switch the Type I sum into a new exponential sum of the form
$$
\sum_n \e\big(f_4(m,n)\big)
$$
where
$$
f_{4}(\nu , n) = (\gamma-1) \gamma^{\gamma/(1-\gamma)} h^{1/(1-\gamma)}\big((n+q)^\gamma - n^{\gamma}\big)^{1/(1-\gamma)}(\nu - \theta)^{\gamma/(\gamma - 1)}
$$
so that we can apply the method of exponent pair. Eventually we prove a upper bound of \eqref{eq:nes} as the same as the bound of \eqref{eq:oes} by Heath-Brown. The readers can compare the assumptions of our Theorem~\ref{thm:eps} with the prerequisite \eqref{eq:HBP} of Heath-Brown's classical result on Piatetski-Shapiro primes. 

\subsection{Comparison to Kumchev \cite{Kumc}}

We also want to mention that our methods are highly related to Balog and Friedlander \cite{BaFr}, Kumchev \cite{Kumc}. Kumchev applied A-process and B-process twice to bound the exponential sum. The reason we did not try the same thing is that the estimation of Kumchev \cite{Kumc} has a minor mistake in his second A-process. The mistake occurs when $q$ is small which is far from the critical range, so it is not clear that the error ruins the whole proof. Since we improve the result by Kumchev, we do not consider to correct the error here.

In Kumchev's paper, the Weyl-van der Corput inequality is applied to obtain the equation (25) with a parameter
\begin{equation}
	\label{eq:ms}
	Q_2 \le FM^{-1}
\end{equation}
with
$$
Q_2 = \lfloor x^{20(1-\gamma) + 24 \delta - 2 + \eps} N^{-1} \rfloor + 1, F = qx^\gamma H N^{-1}
$$
and
\begin{align}
	\label{eq:ms2}
	N \le x^{20(1-\gamma) + 24 \delta - 2 + \eps}. 
\end{align}
Additionally, these parameters satisfy all the conditions 
\begin{enumerate}
	\item $MN\asymp x$, 
	\item $q\leqslant Q = \lfloor x^{2(1-\gamma) + 2\delta + \epsilon}\rfloor  + 1$, 
	\item $H\leqslant x^{1-\gamma + \delta + \epsilon}$,
	\item $FM^{-1}\gg x^{\epsilon}$. 
\end{enumerate}
It is not too hard to analyze that equation \eqref{eq:ms} does not hold all the time. We give an easy case, for example, if we take $q = 1$ and $H = x^{1-\gamma +\epsilon}$, we deduce from \eqref{eq:ms} that
$$
N\geqslant x^{20(1-\gamma) + 24 \delta - 2 + \eps},
$$
which contradicts \eqref{eq:ms2}.

\section{Bounds on exponential sums}
\label{sec:4}

\subsection{Initial approach}
\label{sec:4.1}
In this section we prove Theorem~\ref{thm:eps} by considering sums of the form
$$
\sum_{h \sim H}\delta_h \mathop{\sum_{m\sim M}\sum_{n\sim N}}_{mn\sim x}a_{m}b_{n}\e\big(\alpha mn + h(mn + u)^{\gamma}\big),
$$
where $|\delta_h| \leqslant 1$, $H \leqslant x^{1-\gamma+\eps}$ and $0\leqslant u \leqslant 1$.
If the coefficients $a_{m}$ and $b_{n}$ satisfy the conditions
$$
|a_{m}| \leqslant 1,\quad b_{n} = 1 \quad \text{or}\quad b_{n} = \log n,
$$
we denote the sum by $S_{\uppercase\expandafter{\romannumeral1}}$, and if they satisfy the conditions
$$
|a_{m}| \leqslant 1,\quad |b_{n}|\leqslant 1,
$$ 
we denote it by $S_{\uppercase\expandafter{\romannumeral2}}$.

For $S_{\uppercase\expandafter{\romannumeral2}}$ we use the estimate obtained in Proposition 2 of \cite{BaFr} with a simple adjustment.
\begin{lemma}
	\label{lem:HB1}
	Let $\frac{5}{6}< \gamma <1$ and $N$ satisfies the condition
	$$
	x^{1-\gamma +\eps} \leqslant N \leqslant x^{5\gamma -4-\eps},
	$$
	then 
	$$
	S_{\uppercase\expandafter{\romannumeral2}} \ll x^{1 -\eps}.
	$$
\end{lemma}

For $S_{\uppercase\expandafter{\romannumeral1}}$ we shall prove the following lemma. 
\begin{lemma}
	\label{lem:type1}
	Assume $(k,l)$ is an exponent pair such that
	\begin{equation}
		4k+2l-1>0,  \notag
	\end{equation} 
	and $\gamma,~N$ satisfy
	\begin{equation}
		\frac{5k-l+3}{6k-2l+4}<\gamma<1\notag
	\end{equation}
	with
	\begin{equation}
		N \geqslant \min \big\{ x^{(1-\gamma) + \frac{1}{2} + \eps}, \max\{x^{\frac{4k+6}{4k-2l+1}(1-\gamma) + \frac{2k-1}{4k-2l+1}+ \eps},x^{2(1-\gamma) + \eps}\} \big\} ,\notag
	\end{equation}
	then
	$$
	S_{\uppercase\expandafter{\romannumeral1}} \ll x^{1-\eps}.
	$$
\end{lemma}

\begin{proof}
	We only need to deal with the situation when $b_{n} = 1$, otherwise we can remove $b_{n}$ by a partial integration. Set
	$$
	f_0(m,n) = h(mn + u)^{\gamma}+ \alpha mn.
	$$
	We trivially have
	\begin{align}
		S_{\uppercase\expandafter{\romannumeral1}} = &\sum_{h\sim H}\delta_h\sum_{m\sim M}a_m\sum_{n\sim N}\e \big(f_{0}(m,n)\big) \notag \\
		\leqslant & \sum_{h \sim H}\sum_{m\sim M}\bigg| \sum_{n \sim N}e\big(f_{0}(m,n)\big) \bigg| \label{1}.
	\end{align}
	
	Firstly, we treat \eqref{1} in a simple way. We can calculate for $m\sim M$, $n\sim N$, $MN\asymp x$, $h\sim H$,
	$$
	\frac{\partial^2 f_{0}}{\partial n^2}(m,n) = \gamma(\gamma -1)hm^{2}(mn+u)^{\gamma-2} \asymp Hx^{\gamma}N^{-2}.
	$$
	By Lemma~\ref{lem:GK3}, \eqref{1} and the condition $H\leqslant x^{1-\gamma + \eps}$, we know that
	\begin{align*}
		S_{\uppercase\expandafter{\romannumeral1}} \ll& x^{\gamma/2}H^{3/2}M + x^{1-\gamma/2}H^{1/2} \\ \ll& x^{\eps}(x^{5/2-\gamma}N^{-1} + x^{3/2-\gamma}).
	\end{align*}
	Thus we have $S_{\uppercase\expandafter{\romannumeral1}} \ll x^{1-\eps}$ for all sufficiently small positive $\eps$ provided that
	\begin{align}
		\label{2}
		1>\gamma >\frac{1}{2}
	\end{align}
	and
	\begin{align}
		\label{3}
		N\geqslant x^{(1-\gamma) + \frac{1}{2} + \eps}.
	\end{align}
	Secondly, we can also estimate $S_{\uppercase\expandafter{\romannumeral1}}$ differently. Since
	\begin{align*}
		f_{0}(m,n) =f_{1}(m,n) + O\big(g_1(m,n)\big),
	\end{align*}
	where
	$$
	f_1(m,n) = \alpha mn + h(mn)^{\gamma}+\gamma u h (mn)^{\gamma-1}
	$$
	and
	$$
	g_1(m,n) \ll Hx^{\gamma-2},
	$$
	one has
	$$
	S_{\uppercase\expandafter{\romannumeral1}} = \sum_{h\sim H}\delta_h\sum_{m\sim M}a_m\sum_{n\sim N}\e \big(f_{1}(m,n)\big)\e\big(g_1(m,n)\big).
	$$
	Thus
	\begin{align}
		&S_{\uppercase\expandafter{\romannumeral1}} - \sum_{h\sim H}\delta_h\sum_{m\sim M}a_m\sum_{n\sim N}\e \big(f_{1}(m,n)\big) \notag\\=& \sum_{h\sim H}\delta_h\sum_{m\sim M}a_m\sum_{n\sim N}\e \big(f_{1}(m,n)\big)\Big(\e\big(g_1(m,n)\big)-1\Big)\notag\\
		\ll&(HMN)(Hx^{\gamma -2}) \ll x^{1-\gamma + \eps}.\label{4}
	\end{align}
	The right-hand side of \eqref{4} $\ll x^{1-\eps}$ holds for sufficiently small positive $\eps$ and we obtain that 
	$$
	S_{\uppercase\expandafter{\romannumeral1}} \ll \sum_{h\sim H}\bigg| \sum_{m\sim M}\sum_{n\sim N}a_m\e\big(f_{1}(m,n)\big) \bigg| + x^{1-\eps}.
	$$
	We apply the Cauchy inequality to the sum over $m$ and the Weyl-van der Corput inequality (Lemma~\ref{lem:A}) to the sum over $n$. It follows that
	\begin{align}
		S_{\uppercase\expandafter{\romannumeral1}} \ll & \sum_{h\sim H}M^{1/2}\bigg( \sum_{m \sim M}\bigg| \sum_{n \sim N}\e\big(f_{1}(m,n)\big) \bigg|^2 \bigg)^{1/2} + x^{1-\eps} \notag \\
		\ll & \sum_{h\sim H}M^{1/2}\bigg( \frac{xN}{Q} + \frac{N}{Q}\sum_{1\leqslant q \leqslant Q}\bigg| \sum_{m\sim M}\sum_{\substack{n\sim N}}\e\big(f_{2}(m,n)\big) \bigg| \bigg)^{1/2} + x^{1-\eps} \notag \\
		\ll & \frac{Hx}{Q^{1/2}} + \frac{x^{1/2}}{Q^{1/2}}\sum_{h\sim H}\bigg( \sum_{1\leqslant q \leqslant Q}\bigg| \sum_{m\sim M}\sum_{\substack{n\sim N}}\e\big(f_{2}(m,n)\big) \bigg| \bigg)^{1/2} + x^{1-\eps},\label{5}
	\end{align}
	where $1\leqslant Q \leqslant N$ and 
	\begin{align*}
		f_{2}(m,n) = & f_{1}(m,n+q)-f_{1}(m,n).
	\end{align*}
	If we choose
	$$
	Q = \lfloor x^{2(1-\gamma) + \eps} \rfloor + 1,
	$$
	the first term in the right-hand side of \eqref{5} $\ll x^{1-\eps}$ and $N$ satisfies that
	\begin{align}
		\label{6}
		N \geqslant x^{2(1-\gamma )+ \eps}.
	\end{align}
	Applying partial summations to the sum over $n$ and $m$ successively we find that if $N\geqslant x^{2(1-\gamma) + \eps}$(which happens to be the same as \eqref{6}),
	\begin{align}
		\label{7}
		\bigg| \sum_{m\sim M}\sum_{\substack{n\sim N}}\e\big(f_{2}(m,n)\big) \bigg| \ll \bigg| \sum_{m\sim M}\sum_{\substack{n\sim N}}\e\Big(hm^{\gamma} \big((n+q)^{\gamma} - n^{\gamma}\big) + \alpha q m\Big) \bigg|.
	\end{align}
	We write $\theta = \{\alpha q\}$ (note that $\theta$ does not depend on $m$ and $n$, and $0\leqslant\theta<1$) and we derive from \eqref{7} that
	$$
	\bigg| \sum_{m\sim M}\sum_{\substack{n\sim N}}\e\big(f_{2}(m,n)\big) \bigg| \ll \bigg| \sum_{m\sim M}\sum_{\substack{n\sim N}}\e\big(f_{3}(m,n)\big) \bigg|,
	$$
	where
	$$
	f_{3}(m,n) = hm^{\gamma}\big((n+q)^{\gamma} - n^{\gamma}\big) + \theta m.
	$$
	Hence
	\begin{align}
		\label{8}
		S_{\uppercase\expandafter{\romannumeral1}} \ll \frac{x^{1/2}}{Q^{1/2}}\sum_{h\sim H}\bigg(\sum_{1\leqslant q \leqslant Q}\bigg|\sum_{m\sim M}\sum_{\substack{n\sim N}}\e\big(f_{3}(m,n)\big)\bigg|\bigg)^{1/2} + x^{1-\eps}
	\end{align}
	provided \eqref{6}.
	$f_{3}(m,n)$ satisfies that
	\begin{align}
		\label{9}
		\frac{\partial f_{3}}{\partial m}(m,n) = \gamma h m^{\gamma - 1}\big((n+q)^\gamma - n^{\gamma}\big) + \theta \asymp qHx^{\gamma-1}+\theta
	\end{align}
	and
	\begin{align}
		\label{10}
		\frac{\partial^2 f_{3}}{\partial m^2}(m,n) = \gamma(\gamma-1) h m^{\gamma - 2}\big((n+q)^\gamma - n^{\gamma}\big)\asymp qHx^{\gamma-2}N.
	\end{align}
	Write $m_{\nu} = m_{\nu}(n)$ which is the solution of equation
	\begin{align}
		\label{11}
		\frac{\partial f_{3}}{\partial m}(m_{\nu},n) = \nu,
	\end{align}
	where $\nu$ runs through 
	\begin{align}
		\label{12}
		\bigg[ \frac{\partial f_{3}}{\partial m}(M_{2},n),\frac{\partial f_{3}}{\partial m}(M_{1},n) \bigg]
	\end{align}
	and $m$ runs through $[M_{1},M_{2}] \subset (M, 2M] $. Furthermore, the endpoints of \eqref{12} $\asymp qHx^{\gamma-1} + \theta$ from \eqref{9}. By \eqref{11}, we have that
	\begin{equation}
	\label{eq:mnu}
	m_{\nu} = \gamma^{1/(1-\gamma)} h^{1/(1-\gamma)} \big( (n+q)^\gamma - n^\gamma \big)^{1/(1-\gamma)} (\nu - \theta)^{1/(\gamma-1)}. 
	\end{equation}
	Considering the function
	$$
	f_{4}(\nu,n) = -f_{3}(m_{\nu},n) + \nu m_{\nu}
	$$
	and \eqref{eq:mnu}, we calculate that
	$$
	f_{4}(\nu , n) =  (\gamma-1) \gamma^{\gamma/(1-\gamma)} h^{1/(1-\gamma)}\big((n+q)^\gamma - n^{\gamma}\big)^{1/(1-\gamma)}(\nu - \theta)^{\gamma/(\gamma - 1)}.
	$$
	We change the order of summation and apply Poisson summation formula (Lemma~\ref{lem:B}) to the sum over $m$, getting that
	\begin{align}
		&\sum_{m\sim M}\sum_{\substack{n\sim N}}\e\big(f_{3}(m,n)\big) \notag \\
		= &\sum_{\substack{n\sim N}}\bigg( \e(-1/8)\sum_{\nu}\bigg(\frac{\partial^2 f_{3}}{\partial m^2}(m_{\nu},n) \bigg)^{-1/2}\e\big(-f_{4}(\nu,n)\big) \notag \\
		&+ O\Big(\log (qHx^{\gamma - 1} + 2) + (qHx^{\gamma}N^{-1})^{-1/2}M\Big) \bigg)\notag\\
		\ll & \bigg|\sum_{\nu}\sum_{n\sim N}\bigg(\frac{\partial^2 f_{3}}{\partial m^2}(m_{\nu},n) \bigg)^{-1/2}\e\big(f_{4}(\nu,n)\big)\bigg|\notag \\ &+ N\log x + q^{-\frac{1}{2}}H^{-\frac{1}{2}}x^{\frac{1}{2}(1-\gamma) + \frac{1}{2}}N^{\frac{1}{2}}.\label{13}
	\end{align}
	Inserting \eqref{13} to \eqref{8} gives that the second and third terms of the right-hand side of \eqref{13} are admissible when
	\begin{align}
		\label{14}
		N \leqslant x^{-2(1-\gamma)+1-\eps}.
	\end{align}
	To deal with the first term in \eqref{13}, we remove the factor $\big(\frac{\partial^2 f_{3}}{\partial m^2}(m_{\nu},n) \big)^{-1/2}$ by the Abel summation. By \eqref{10} and \eqref{eq:mnu}, we have 
	\begin{align*}
		\frac{\partial^2 f_{3}}{\partial m^2}(m_{\nu},n) &= \gamma(\gamma-1) h m_{\nu}^{\gamma - 2}\big((n+q)^\gamma - n^{\gamma}\big)\\
		&= (\gamma-1)\gamma^{1/(\gamma-1)}h^{1/(\gamma-1)}\big((n+q)^\gamma - n^{\gamma}\big)^{1/(\gamma - 1)}(\nu - \theta)^{(2-\gamma)/(1-\gamma)},
	\end{align*}
	which implies that $(\frac{\partial^2 f_{3}}{\partial m^2}(m_{\nu},n) )^{-1/2}$ is monotonic in $n$ and $\nu$. By Abel summation and \eqref{10}, we deduce that
	\begin{align}
		&\sum_{\nu}\sum_{n}\bigg(\frac{\partial^2 f_{3}}{\partial m^2}(m_{\nu},n) \bigg)^{-1/2}\e\big(f_{4}(\nu,n)\big) \notag \\
		\ll & (qHx^{\gamma-2}N)^{-1/2}\bigg|\sum_{\nu}\sum_{n \sim N}\e\big(f_{4}(\nu,n)\big)\bigg|.\label{15}
	\end{align}
	Combining \eqref{8}, \eqref{13} and \eqref{15}, we know that
	\begin{align}
		\label{16}
		S_{\uppercase\expandafter{\romannumeral1}} \ll \frac{x^{(1-\gamma)/4 + 3/4}}{Q^{1/2}H^{1/4}N^{1/4}}\sum_{h\sim H}\bigg(\sum_{1\leqslant q \leqslant Q}q^{-1/2}\bigg|\sum_{\nu}\sum_{n \sim N}\e\big(f_{4}(\nu,n)\big)\bigg|\bigg)^{1/2} + x^{1-\eps}
	\end{align}
	provided \eqref{6} and \eqref{14}. 
	Next, we will estimate the sum over $n$ by the method of exponent pair. By well-known Faà di Bruno's formula giving the $r$-th derivative of composition functions, it follows that
	$$
	\frac{\partial^{r} f_{4}}{\partial n^{r}}(\nu ,n) \asymp qHx^{\gamma}N^{-1-r},
	$$
	we have
	$$
	\sum_{n}\e\big(f_{4}(\nu,n)\big) \ll (qHx^{\gamma}N^{-2})^{k}N^{l} + (qHx^{\gamma}N^{-2})^{-1},
	$$
	where $(k,l)$ is an exponent pair. The second term of this estimate can be omitted when
	$$
	N\leqslant x^{-\frac{2}{3}(1-\gamma) + 1 - \eps},
	$$
	which is weaker than \eqref{14}.
	Substituting this estimate in \eqref{16}, we find that
	$$
	S_{\uppercase\expandafter{\romannumeral1}}\ll x^{1-\eps}
	$$
	holds when
	$$
	N^{4k-2l+1}\geqslant x^{(4k+6)(1-\gamma) + (2k-1) + \eps},
	$$
	additionally with \eqref{6} and \eqref{14}. To get the lower bound of $N$, we need to choose exponent pairs satisfying
	\begin{align}
		\label{17}
		4k-2l+1>0.
	\end{align}
	Hence 
	$$
	S_{\uppercase\expandafter{\romannumeral1}}\ll x^{1-\eps}
	$$
	provided \eqref{17},
	\begin{align}
		\label{18}
		\frac{5k-l+3}{6k-2l+4}<\gamma<1
	\end{align}
	and
	\begin{align}
		\label{19}
		\max\{x^{\frac{4k+6}{4k-2l+1}(1-\gamma) + \frac{2k-1}{4k-2l+1}+\eps}, x^{2(1-\gamma) + \eps}\}\leqslant N\leqslant x^{-2(1-\gamma) + 1 - \eps}.
	\end{align}
	\eqref{18} is to make sure the range of \eqref{19} exists. 
	
	In conclusion, if one of the conditions
	
	(i) \eqref{2} and \eqref{3},
	
	(ii) \eqref{17}, \eqref{18} and \eqref{19}\\
	holds, then
	$$
	S_{\uppercase\expandafter{\romannumeral1}}\ll x^{1-\eps}.
	$$
	If \eqref{17} and \eqref{18} hold, we note that
	\begin{align*}
	(1-\gamma) + \frac{1}{2} < -2(1-\gamma) + 1,
	\end{align*}
	which means we can put the range of \eqref{3} and \eqref{19} together and finish the proof of Lemma~\ref{lem:type1}
%
\end{proof}

Eventually, we combine our better exponent sum estimation by the Heath-Brown identity \cite{HB2}. 
\begin{lemma}
	\label{lem:HB2}
	Let $3\leqslant U<V<Z<x$ and suppose that $Z-\frac{1}{2}\in \mathbb{N}$, $x\geqslant 64Z^{2}U,~Z\geqslant4U^{2},~V^{3}\geqslant 32x$. Assume further that $f(n) = 0$ when $n\leqslant x$ or $n > 2x$ and that $|f(n)| \leqslant f_{0}$ otherwise. Then the sum
	$$
	\sum_{n \sim x}\Lambda(n)f(n)
	$$
	may be decomposed into $O(\log^{10}x)$ sums, each either of Type \uppercase\expandafter{\romannumeral1} with $N>Z$, or of Type \uppercase\expandafter{\romannumeral2} with $U<N<V$.
\end{lemma}

We apply Lemma~\ref{lem:HB2} with $$f(n) = \sum\limits_{h\sim H}\delta_h \e\big(\alpha n + h(n + u)^{\gamma}\big),$$ $U = 2^{-10}x^{1-\gamma +\eps}$, $V= 4x^{1/3}$ and
$$
Z=\max \big( \lfloor x^{\frac{4k+6}{4k-2l+1}(1-\gamma) + \frac{2k-1}{4k-2l+1}+\eps}\rfloor, \lfloor 5x^{1/3} \rfloor, \lfloor x^{2(1-\gamma) + \eps}\rfloor \big) + \frac{1}{2}.
$$
Note that $x\geqslant 64Z^{2}U$ implies
\begin{align*}
	\frac{12k+10}{12k-2l+13}<\gamma<1
\end{align*}
and this condition is stronger than \eqref{18} when \eqref{17} holds. In order to apply Lemma \ref{lem:HB1}, we require the additional condition $V\leqslant x^{5\gamma-4-\epsilon}$, which implies
$$
\frac{13}{15}<\gamma<1.
$$
Then Theorem~\ref{thm:eps} follows from Lemma~\ref{lem:HB1} and Lemma~\ref{lem:type1}.

\section{The ternary Goldbach problem on Piatetski-Shapiro primes}
\label{sec:5}
We prove Theorem \ref{thm:3ps} via Theorem \ref{thm:wb} and a key lemma proved by Sun, Du and Pan \cite{SDP}. 
\subsection{Proof of Theorem \ref{thm:wb}}
The following proof is similar to Balog's argument in Theorem~4 of \cite{BaFr} with slight changes to the choice of parameters. The case $\alpha = 0$ is simply the Piatetski-Shapiro prime number theorem; see \cite{RiSa}. Let
$$
T(\alpha) = \frac{1}{\gamma}\sum_{\substack{p \le N \\ p \in \mathcal{N}^{(c)}}} p^{1-\gamma}\log p \cdot \e(\alpha p).
$$ 
By Lemma \ref{lem:PS}, it follows that
\begin{align*}
T(\alpha) &=  \frac{1}{\gamma}\sum_{p \le N} \e(\alpha p) p^{1-\gamma}\log p \big( \fl{-p^\gamma} - \fl{-(p+1)^\gamma} \big)\\
&= S_1 + S_2 + O(\log N)
\end{align*}
where
$$
S_1 = \sum_{p \le N} \log p\cdot \e(\alpha p)
$$
and
$$
S_2 = \frac{1}{\gamma} \sum_{p \le N} \e(\alpha p) p^{1-\gamma} \log p \big( \psi(-(p+1)^\gamma) - \psi(-p^\gamma) \big)
$$
It is sufficient to show that
$$
S_2 \ll N^{1-\eps}. 
$$
Using integration by parts, we have
$$
S_2 = \frac{1}{\gamma} \sum_{n \le N} \Lambda(n) \e(\alpha n) n^{1-\gamma} \big( \psi(-(p+1)^\gamma) - \psi(-p^\gamma) \big) + O(N^{3/2-\gamma})
$$
in which the error term is admissible for $\gamma > 1/2 + \eps$. Applying a splitting argument, it is sufficient to prove
$$
 \sum_{n \sim x} \Lambda(n) \e(\alpha n) n^{1-\gamma} \big( \psi(-(p+1)^\gamma) - \psi(-p^\gamma) \big)\ll x^{1-\eps}
$$
when $x \leqslant N$.
By the Vaaler's approximation (Lemma \ref{lem:Vaaler}), we break the sum into
$$
S_3 + O(S_4)
$$
where
$$
S_3 = \frac{1}{\gamma} \sum_{0<|h|\le H_{0}} a_h \sum_{n \sim x} \Lambda(n) \e(\alpha n) n^{1-\gamma} \big(\e(h(n+1)^\gamma) - \e(hn^\gamma) \big)
$$
and
$$
S_4 = \sum_{|h|\le H_{0}} b_h \sum_{n \sim x} \Lambda(n) \e(\alpha n) n^{1-\gamma} \big(\e(h(n+1)^\gamma) + \e(hn^\gamma) \big)
$$
in which $H_{0} = x^{1-\gamma + \eps}$. 

Firstly, we prove $S_3 \ll x^{1-\eps}$. It is sufficient to prove that
$$
\sum_{h \sim H} h^{-1} \bigg| \sum_{n \sim x} \Lambda(n) \e(\alpha n) n^{1-\gamma} \big(\e(h(n+1)^\gamma) - \e(hn^\gamma) \big) \bigg| \ll x^{1-\eps} 
$$
for $H\leqslant x^{1-\gamma +\epsilon}$. Writing that
$$
\e(h(n+1)^\gamma) - \e(hn^\gamma) = 2\pi i \gamma h \int_0^1 (n+u)^{\gamma-1} \e(h(n+u)^\gamma) \, du. 
$$
By a partial summation, it is sufficient to prove that
$$
\sum_{h \sim H} \bigg|\sum_{n \sim x} \Lambda(n) \e\big(\alpha n + h(n+u)^\gamma\big) \bigg| \ll x^{1-\eps}
$$
for $0 \le u \le 1$ and $H\leqslant x^{1-\gamma+\epsilon}$, which is done by Corollary~\ref{cor:eps} for
\begin{align*}
	\frac{498}{596}<\gamma<1
\end{align*}
and uniformly for $\alpha$. 

Eventually, we consider $S_4$. The contribution from $h=0$ in $S_4$ is 
$$
\ll H_{0}^{-1} \sum_{n \sim x} \Lambda(n) \e(\alpha n) n^{1-\gamma} \ll x^{1- \eps}. 
$$
The contribution from $h \neq 0$ in $S_4$ can be bounded as $S_{3}$.

\subsection{Proof of Theorem \ref{thm:3ps}}

 If the error term in \eqref{eq:BF} becomes $O(N/\log^A N)$ for any positive number $A \ge 2$, we call that $c$ satisfies the weak Balog-Friedlander condition. We mention a useful proposition, which is \cite[Corollary 3.2]{SDP}.

\begin{lemma}
\label{lem:SDP}
Suppose that $c_1, c_2, c_3 \in (1, \frac{6}{5})$ satisfy the weak Balog-Friedlander condition. Then every sufficiently large odd $N$ can be represented as 
$$
N = p_1 + p_2 + p_3
$$
where $p_i \in \cN^{(c_i)}$ for $1 \le i \le 3$ with \eqref{eq:PS3} holds and the singular series $\fS$ is defined as \eqref{eq:ss}. 
\end{lemma}

It is trivial that the Balog-Friedlander condition implies the weak one. Applying Lemma~\ref{lem:SDP} and combining Theorem~\ref{thm:wb}, we complete the proof of Theorem~\ref{thm:3ps}. 
	
\section{Piatetski-Shapiro primes in arithmetic progressions}
\label{sec:6}
The construction to this problem is written with details in \cite{GLZ2}. We only give the main ideas and the proof of the core part. 

By Lemma \ref{lem:PS}, we have
\begin{equation}\label{pi-expan}
 \pi^{(c)}(x;q,a)=\Sigma_1(x)+\Sigma_2(x)+O(1),
\end{equation}
where
\begin{equation*}
\Sigma_1(x) =\gamma\sum_{\substack{p\leqslant x\\ p\equiv a \pmod q}}p^{\gamma-1},
\end{equation*}
\begin{equation*}
\Sigma_2(x) =\sum_{\substack{p\leqslant x\\ p\equiv a \pmod q}}
\big(\psi(-(p+1)^\gamma)-\psi(-p^\gamma)\big).
\end{equation*}
By using a partial summation, it gives that
$$
\Sigma_1(x) = \gamma x^{\gamma-1}\pi(x;q,a)-\gamma(\gamma-1)\int_2^xu^{\gamma-2}\pi(u;q,a)\mathrm{d}u.
$$
Next, we turn our attention to $\Sigma_2(x)$. By a partial summation and a splitting argument, we only need to give the upper bound estimate of the following sum 
$$
\mathcal{S}=\sum_{\substack{n \sim N \\ n\equiv a \pmod q}} \Lambda(n)\big(\psi(-(n+1)^\gamma)-\psi(-n^\gamma)\big),
$$
for $N \le x$ since we have $\Sigma_2(x) \ll \cS \log x $. 
According to Vaaler's approximation, i.e. Lemma \ref{lem:Vaaler}, we write
\begin{equation}\label{S-decompo}
\mathcal{S}=\mathcal{S}_1+O(|\mathcal{S}_2|),
\end{equation}
where
\begin{equation*}
\mathcal{S}_1=\sum_{\substack{n \sim N \\ n\equiv a \pmod q}} \Lambda(n)
\sum_{0<|h|\leqslant H}a_h\big(\mathbf{e}(h(n+1)^\gamma)-\mathbf{e}(hn^\gamma)\big),
\end{equation*}
\begin{equation*}
\mathcal{S}_2=\sum_{\substack{n\sim N \\ n\equiv a \pmod q}} \Lambda(n)
\sum_{|h|\leqslant H}b_h\big(\mathbf{e}(h(n+1)^\gamma)+\mathbf{e}(hn^\gamma)\big),
\end{equation*}
with $H = x^{1-\gamma + \eps}$. 

Firstly, we shall consider the upper bound of $\mathcal{S}_1$. By a partial summation (see \cite{GLZ2} for details), we have 
\begin{equation}
\label{eq:S1}
 \mathcal{S}_1 \ll
x^{\gamma-1}\cdot\max_{t \sim N}\sum_{0<|h|\leqslant H}  \Bigg|\sum_{\substack{N<n\leqslant t\\ n\equiv a \pmod q}}  \Lambda(n)\mathbf{e}(hn^\gamma)\Bigg|.
\end{equation}
By using the well--known orthogonality
\begin{equation*}
 \frac{1}{q}\sum_{m=1}^q\mathbf{e}\bigg(\frac{(n-a)m}{q}\bigg)=
 \begin{cases}
   1, & \textrm{if $q|n-a$}, \\
   0, & \textrm{if $q\nmid n-a$},
 \end{cases}
\end{equation*}
we can represent the innermost sum in \eqref{eq:S1} as
\begin{equation}\label{inner-trans}
  \sum_{\substack{N<n\leqslant t\\ n\equiv a \pmod q}}\Lambda(n)\mathbf{e}(hn^\gamma)
  =\frac{1}{q}\sum_{m=1}^q\sum_{N<n\leqslant t}\Lambda(n)\mathbf{e}\bigg(hn^\gamma+\frac{(n-a)m}{q}\bigg).
\end{equation}
From (\ref{eq:S1}) and (\ref{inner-trans}), we know that it suffices to estimate
\begin{equation*}
  \max_{t \sim N}\sum_{0<|h|\leqslant H}
  \Bigg|\sum_{N<n\leqslant t}\Lambda(n)\mathbf{e}(hn^\gamma+nmq^{-1})\Bigg|.
\end{equation*}
Applying Corollary~\ref{cor:eps}, the bound of $\cS_1$ is satisfied.

Now, we focus on the upper bound of $\mathcal{S}_2$. The contribution from $h=0$ is
$$
2b_0\sum_{\substack{N < n \leqslant x \\ n\equiv a \pmod q}} \Lambda(n)\ll\frac{b_0x}{\varphi(q)} \ll xH^{-1},
$$
The contribution from $h\neq0$ can be bounded as $\cS_1$. 

\section{Piatetski-Shapiro primes in a Beatty sequence}
\label{sec:7}
The construction of this problem is partially similar to the proof of Theorem~\ref{thm:AP}. Also since we have enough details in~\cite{GLZ}, we omit the details here. 

By \cite[Section 4]{GLZ}, let
$$
H = x^\eps \mand J = x^{1-\gamma + \eps}.
$$ 
It is sufficient to prove that
\begin{equation}\label{eq:cT2}
\max_{x/2<t\leqslant x}  \sum_{0 < |h| \le H} |h|^{-1} \sum_{0 < j \le J} |\cT_4| \ll x^{1-\varepsilon}, \end{equation}
where
$$
\cT_4=\sum_{x/2<n\le t}\Lambda(n)\cdot\e\(jn^\gamma + \omega h n-\omega h \beta ) \)
$$
and $\omega = \alpha^{-1}$.

Applying Corollary~\ref{cor:eps} to \eqref{eq:cT2}, the proof is done. 

\section{A generalized Balog-Friedlander condition}
\label{sec:8}

In this section, for historical context, we extend the idea of Section~\ref{sec:4.1} and show the following Theorem.
\begin{theorem}
	\label{thm:delta}
	Let $(k,l)$ be an exponent pair such that
	$$
	4k-2l+1>0 
	$$
	and $\gamma$, $0\leqslant \delta\leqslant 1-\gamma$ subjected to
	\begin{align*}
		\frac{12k-2l+13}{-2l+3}(1-\gamma)+\frac{20k-4l+16}{-2l+3}\delta<1
	\end{align*}
	and
	\begin{align*}
		\frac{15}{2}(1-\gamma)+9\delta<1.
	\end{align*}
	Assume further that $H\leqslant x^{1-\gamma + \delta + \eps}$ and $0\leqslant u \leqslant 1$. Then
	\begin{align*}
		\min\biggl\{1,\frac{x^{1-\gamma}}{H}\biggr\}\sum_{h\sim H}\bigg| \sum_{n \sim x}\Lambda(n)\e\big(\alpha n + h(n+u)^{\gamma}\big) \bigg| \ll x^{1-\delta -\eps}.
	\end{align*}
	\end{theorem}
Setting $\delta=0$ in Theorem~\ref{thm:delta} yields Theorem~\ref{thm:eps}. Previously, it was common to study parameters $\gamma$, $\delta$ satisfying a relation of the form
$$
A(1-\gamma) + B\delta < 1,
$$
with the order being $O(x^{1-\delta-\epsilon})$. While Sun, Du and Pan's work \cite{SDP} simplifies the problem to the $\delta=0$ case, we prove the broader theorem to align with the traditional perspective. Similar to the proof of Theorem \ref{thm:3} and also Page 50 in \cite{BaFr}, we also obtain a generalized but more traditional Balog-Friedlander condition. We omit the proof from Theorem \ref{thm:delta} to Theorem \ref{thm:GBF}. 

\begin{theorem}
	\label{thm:GBF}
	Let $(k,l)$ be an exponent pair such that
	$$
	4k-2l+1>0 
	$$
	and $\gamma$, $0\leqslant \delta\leqslant 1-\gamma$ subjected to
	\begin{align*}
		\frac{12k-2l+13}{-2l+3}(1-\gamma)+\frac{20k-4l+16}{-2l+3}\delta<1
	\end{align*}
	and
	\begin{align*}
		\frac{15}{2}(1-\gamma)+9\delta<1.
	\end{align*}
	For every sufficiently large $N$, it follows that
$$
		\sum_{\substack{p \le N \\ p \in \mathcal{N}^{(c)}}} c p^{1-\frac{1}{c}}\log p \cdot \e(\alpha p) = \sum_{p \le N} \log p\cdot \e(\alpha p) + O ( N^{1-\delta-\eps} )
$$
	uniformly for $\alpha$. 
\end{theorem}

Now we give the proof of Theorem \ref{thm:delta}. Consider sums of the form
$$
\min\biggl\{1,\frac{x^{1-\gamma}}{H}\biggr\}\sum_{h \sim H}\delta_h \mathop{\sum_{m\sim M}\sum_{n\sim N}}_{mn\sim x}a_{m}b_{n}\e\big(\alpha mn + h(mn + u)^{\gamma}\big),
$$
where $|\delta_h| \leqslant 1$, $H \leqslant x^{1-\gamma+\delta+\eps}$ and $0\leqslant u \leqslant 1$.
If the coefficients $a_{m}$ and $b_{n}$ satisfy the conditions
$$
|a_{m}| \leqslant 1,\quad b_{n} = 1 \quad \text{or}\quad b_{n} = \log n,
$$
we denote the sum by $S'_{\uppercase\expandafter{\romannumeral1}}$, and if they satisfy the conditions
$$
|a_{m}| \leqslant 1,\quad |b_{n}|\leqslant 1,
$$ 
we denote it by $S'_{\uppercase\expandafter{\romannumeral2}}$.

Following the Proposition 2 of \cite{BaFr}, we have
\begin{lemma}
	\label{lem:HB1 2}
	Let $\gamma$, $0\leqslant \delta\leqslant 1-\gamma$ and $N$
	satisfies the condition
	$$6(1-\gamma) + 8\delta <1$$
	and
	$$
	x^{1-\gamma +2\delta +\eps} \leqslant N \leqslant x^{5\gamma -4-6\delta-\eps},
	$$
	then 
	$$
	S'_{\uppercase\expandafter{\romannumeral2}} \ll x^{1 -\delta -\eps}.
	$$
\end{lemma}
For $S'_{I}$, we will prove the following lemma.
\begin{lemma}
	\label{lem:type1 2}
	Assume $(k,l)$ is an exponent pair such that
	\begin{equation}
		4k+2l-1>0,  \notag
	\end{equation} 
	and $\gamma,~0\leqslant \delta \leqslant 1-\gamma,~N$ satisfy
	\begin{equation}
		\frac{12k-4l+8}{2k-2l+2}(1-\gamma) + \frac{14k-4l+9}{2k-2l+2}\delta < 1 \notag
	\end{equation}
	with
	\begin{equation}
		N \geqslant \min \big\{ x^{(1-\gamma) + \frac{1}{2} + \frac{3}{2}\delta + \eps}, \max\{x^{\frac{4k+6}{4k-2l+1}(1-\gamma) + \frac{2k-1}{4k-2l+1}+ \frac{6k+7}{4k-2l+1}\delta+ \eps},x^{2(1-\gamma) +3\delta+ \eps}\} \big\} ,\notag
	\end{equation}
	then
	$$
	S'_{\uppercase\expandafter{\romannumeral1}} \ll x^{1-\delta-\eps}.
	$$
\end{lemma}
\begin{proof}
	Firstly, we estimate $S'_{I}$ simply. Following the argument similar to \eqref{1} and Lemma~\ref{lem:GK3}, we have
	\begin{align*}
		S'_{I}\ll \min\biggl\{1,\frac{x^{1-\gamma}}{H}\biggr\} (x^{\gamma/2+1}H^{3/2}N^{-1} + x^{1-\gamma/2}H^{1/2}),
	\end{align*}
	which $\ll x^{1-\delta-\eps}$ when
	$$
	2(1-\gamma)+2\delta<1
	$$
	and
	\begin{align}
		\label{equ:me1}
		N\geqslant x^{(1-\gamma)+\frac{1}{2}+\frac{3}{2}\delta+\eps}.
	\end{align}
	
	Secondly, we estimate $S'_{I}$ differently. A similar approach to \eqref{4} yields
	\begin{align}
		S'_{\uppercase\expandafter{\romannumeral1}} \ll\min\biggl\{1,\frac{x^{1-\gamma}}{H}\biggr\} \sum_{h\sim H}\bigg| \sum_{m\sim M}\sum_{n\sim N}a_m\e\big(f_{1}(m,n)\big) \bigg| + x^{1-\delta-\eps} \label{equ:f1}
	\end{align}
	provided
	$$
	(1-\gamma) + 2\delta<1.
	$$
	Applying Lemma~\ref{lem:A} to the first term in the right-hand side of \eqref{equ:f1} over $n$ and taking parameter
	$$
	Q'=\fl{x^{2(1-\gamma)+2\delta+\eps}}+1,
	$$ 
	we get
	\begin{align}
		S'_{I}\ll \min\biggl\{1,\frac{x^{1-\gamma}}{H}\biggr\}\frac{x^{1/2}}{Q'^{1/2}}\sum_{h\sim H}\bigg( \sum_{1\leqslant q \leqslant Q'}\bigg| \sum_{m\sim M}\sum_{\substack{n\sim N}}\e\big(f_{2}(m,n)\big) \bigg| \bigg)^{1/2} + x^{1-\delta+\eps} \label{equ:f2}
	\end{align}
	provided
	$$
	N\geqslant x^{2(1-\gamma)+2\delta+\eps}.
	$$
	Moreover, from partial summation over $m$ and $n$, we have
	\begin{align}
		\bigg| \sum_{m\sim M}\sum_{\substack{n\sim N}}\e\big(f_{2}(m,n)\big) \bigg| \ll \bigg| \sum_{m\sim M}\sum_{\substack{n\sim N}}\e\big(f_{3}(m,n)\big) \bigg| \label{equ:f23}
	\end{align}
	under the condition
	$$
	N\geqslant x^{2(1-\gamma)+3\delta+\eps},
	$$
	which constitutes one of the final constraints in $N$.
	From Poisson summation formula (Lemma~\ref{lem:B}) (i.e., \eqref{9}-\eqref{13}) and partial summation, we obtain
	\begin{align}
		\bigg| \sum_{m\sim M}\sum_{\substack{n\sim N}}\e\big(f_{3}(m,n)\big) \bigg| \ll &(qHx^{\gamma-2}N)^{-\frac{1}{2}}\bigg| \sum_{\nu}\sum_{\substack{n\sim N}}\e\big(f_{4}(m,n)\big) \bigg|\notag \\&+ N\log{x} + q^{-\frac{1}{2}}H^{-\frac{1}{2}}x^{\frac{1}{2}(1-\gamma)+\frac{1}{2}}N^{\frac{1}{2}}.\label{equ:f34}
	\end{align}
	Note that the second and third terms in the right side of \eqref{equ:f34} vanish when
	$$
	N\leqslant x^{-2(1-\gamma)+1-2\delta-\eps}.
	$$
	Using exponent pair method to the sum over $n$, one has
	\begin{align}
		\sum_{n}\e\big(f_{4}(\nu,n)\big) \ll (qHx^{\gamma}N^{-2})^{k}N^{l} + (qHx^{\gamma}N^{-2})^{-1}\label{equ:f4}
	\end{align}
	and the second term in the right side of \eqref{equ:f4} is admissible when
	$$
	N\leqslant x^{-\frac{2}{3}(1-\gamma)+1-\frac{2}{3}\delta-\eps}.
	$$
	Combining \eqref{equ:f2}, \eqref{equ:f23}, \eqref{equ:f34} and \eqref{equ:f4} with the condition that $(k,l)$ satisfies
	$$
	4k-2l+1>0,
	$$ 
	we calculate
	$$
	S'_{I}\ll x^{1-\delta-\eps}
	$$
	provided
	$$
	\frac{12k-4l+8}{2k-2l+2}(1-\gamma) + \frac{14k-4l+9}{2k-2l+2}\delta < 1 
	$$
	and
	\begin{align}
		\label{equ:me2}
		x^{-2(1-\gamma)+1-2\delta-\eps}\geqslant N \geqslant  \max\{x^{\frac{4k+6}{4k-2l+1}(1-\gamma) + \frac{2k-1}{4k-2l+1}+ \frac{6k+7}{4k-2l+1}\delta+ \eps},x^{2(1-\gamma) +3\delta+ \eps}\}.
	\end{align}
	
	Finally, putting the range of \eqref{equ:me1} and \eqref{equ:me2} together finish the proof of Lemma~\ref{lem:type1 2}.
\end{proof}
	Next, we prove Theorem~\ref{thm:delta}. We apply Lemma~\ref{lem:HB2} with $$f(n) = \sum\limits_{h\sim H}\delta_h \e\big(\alpha n + h(n + u)^{\gamma}\big),$$ $U = 2^{-10}x^{1-\gamma +2\delta +\eps}$, $V= 4x^{1/3}$ and
	$$
	Z=\max \big( \lfloor x^{\frac{4k+6}{4k-2l+1}(1-\gamma) + \frac{2k-1}{4k-2l+1}+ \frac{6k+7}{4k-2l+1}\delta+\eps}\rfloor, \lfloor 5x^{1/3} \rfloor, \lfloor x^{2(1-\gamma) + 4\delta+ \eps}\rfloor \big) + \frac{1}{2}.
	$$
	Lemma~\ref{lem:HB1 2} and Lemma~\ref{lem:type1 2} imply that all assumptions of Lemma~\ref{lem:HB2} are satisfied, with the exception of the conditions $x\geqslant64Z^{2}U$ and $V\leqslant x^{5\gamma-4-6\delta-\eps}$, which is calculated to be
	\begin{align*}
		\frac{12k-2l+13}{-2l+3}(1-\gamma)+\frac{20k-4l+16}{-2l+3}\delta<1
	\end{align*}
	and
	\begin{align*}
		\frac{15}{2}(1-\gamma)+9\delta<1.
	\end{align*}

\section*{Acknowledgement}
This work was supported by the National Natural Science Foundation of China (No. 11901447, 12271422), the Natural Science Foundation of Shaanxi Province (No. 2024JC-YBMS-029). and the Shaanxi Fundamental Science Research Project for Mathematics and Physics (No. 22JSY006).


\begin{thebibliography}{99}

\bibitem{BaFr}
A. Balog and J. Friedlander,
A hybrid of theorems of Vinogradov and Piatetski-Shapiro.
\emph{Pacific J. Math.} 156 (1992), no. 1, 45--62.

\bibitem{BBBSW}
R.~C.~Baker, W.~D.~Banks, J.~Br\"{u}dern, I.~E.~Shparlinski and A.~J.~Weingartner,
Piatetski-Shapiro sequences.
\emph{Acta Arith.} 157 (2013), no. 1, 37--68.

\bibitem{Bour}
J. Bourgain, 
Decoupling, exponential sums and the Riemann zeta function. 
\emph{J. Amer. Math. Soc.} 30 (2017), no. 1, 205--224. 

\bibitem{GraKol}
S.~W.~Graham and G.~Kolesnik,
Van der Corput's method of exponential sum.
London Mathematical Society Lecture Note Series, 126.
Cambridge University Press, Cambridge, 1991.

\bibitem{Guo1}
V.~Z.~Guo,
Piatetski-Shapiro primes in a Beatty sequence.
\emph{J. \ Number Theory} 156 (2015), 317--330. 

\bibitem{Guo2}
V. Z. Guo,  
Almost primes in Piatetski-Shapiro sequences. 
\emph{AIMS Math.} 6 (2021), no. 9, 9536--9546.

\bibitem{GLZ}
V. Z. Guo,  J. Li and M. Zhang, 
Piatetski-Shapiro primes in the intersection of multiple Beatty sequences. 
\emph{Rocky Mountain J. Math.} 52 (2022), no. 4, 1375--1394.  

\bibitem{GLZ2}
V. Z. Guo,  J. Li and M. Zhang, 
Piatetski-Shapiro primes in arithmetic progressions. 
\emph{Ramanujan J.} 60 (2023), no. 3, 677--692.

\bibitem{HB2}
D.~R.~Heath-Brown,
The Pjatecki\u{i}-\u{S}apiro prime number theorem. 
\emph{J.\ Number Theory} 16 (1983), 242--266.

\bibitem{Kumc}
A. Kumchev,  
On the Piatetski-Shapiro-Vinogradov theorem. 
\emph{J. Théor. Nombres Bordeaux} 9 (1997), no. 1, 11--23.

\bibitem{PS}
I.~I.~Piatetski-Shapiro, 
On the distribution of prime numbers in the sequence of the form $\fl{f(n)}$.
\emph{Mat.\ Sb.} 33 (1953), 559--566.

\bibitem{Riva}
J. Rivat, 
Autour d'un théorème de Piatetski-Shapiro (Nombres premiers dans la suite $[n^c]$). 
Thèse de Doctorat, Université de Paris-Sud, 1992.

\bibitem{RiSa}
J.~Rivat and S.~Sargos,
Nombres premiers de la forme $\fl{n^c}$.
\emph{Canad.\ J.\ Math.} 53 (2001), no.~2, 414--433.

\bibitem{RiWu}
J.~Rivat and J.~Wu,
Prime numbers of the form $\fl{n^c}$.
\emph{Glasg.\ Math.\ J.} 43 (2001), no.~2, 237--254. 

\bibitem{SDP}
Y. Sun, S. Du and H. Pan,
Vinogradov's theorem with Piatetski-Shapiro primes.
\emph{Int. Math. Res. Not. IMRN} 2025, no. 15, rnaf125.

\bibitem{Vaal}
J.~D.~Vaaler,
Some extremal problems in Fourier analysis. 
\emph{Bull.\ Amer.\ Math.\ Soc.} 12 (1985), 183--216.

\end{thebibliography}
\end{document}